\documentclass[12pt,reqno,a4paper]{amsart}
\usepackage{verbatim,rcs,cite}
\usepackage{amssymb,amsfonts,latexsym,mathrsfs}
\usepackage{amsmath}
\usepackage{amsthm}
\usepackage{mathrsfs}
\usepackage[all]{xy}
\usepackage{pstricks}
\usepackage{hyperref}

\newcommand{\tmop}[1]{\ensuremath{\operatorname{#1}}}
\newcommand{\assign}{:=}
\newcommand{\nin}{\not\in}

\newtheorem{theorem}{Theorem}[section]

\newtheorem{proposition}[theorem]{Proposition}
\newtheorem{corollary}[theorem]{Corollary}

\theoremstyle{remark}
\newtheorem{example}[theorem]{Example}
\newtheorem{remark}[theorem]{Remark}
\theoremstyle{definition}
\newtheorem{definition}[theorem]{Definition}

\title{Homology representations of unitary reflection groups}

\date{}

\author{Justin Koonin}

\dedicatory{\upshape
School of Mathematics and Statistics \\
University of Sydney, NSW 2006, Australia\\[.5em]
{\it E-mail address:} \ \texttt{justin.koonin@sydney.edu.au} }

\begin{document}

\thanks{\noindent{AMS subject classification (2010): 20F55, 05E18 }}


\thanks{\noindent{Keywords: Poset topology, unitary reflection groups, homology representations }}

\thanks{\noindent{The author is supported by ARC Grant \#DP110103451 at the University of Sydney.}}

\begin{abstract}

This paper continues the study of the poset of eigenspaces of elements of a unitary reflection group (for a fixed eigenvalue), which was commenced in \cite{Koonin2012-5} and \cite{Koonin2012-2}. The emphasis in this paper is on the representation theory of unitary reflection groups. The main tool is the theory of poset extensions due to Segev and Webb (\cite{SeWe1994}).  The new results place the well-known representations of unitary reflection groups on the top homology of the lattice of intersections of hyperplanes into a natural family, parameterised by eigenvalue.

\end{abstract}

\maketitle

\section{Introduction}

Let $V$ be a complex vector space of finite dimension, and $G \subseteq GL(V)$ a unitary reflection group in $V$.  Denote by $\mathcal{A}(G)$ the set of
reflecting hyperplanes of all reflections in $G$, and
$\mathcal{M}_{\mathcal{A}(G)}$ the hyperplane complement -- that is, the smooth
manifold which remains when all the reflecting hyperplanes are removed from
$V$. There is an extensive literature studying the topology of
$\mathcal{M}_{\mathcal{A}(G)}$ (\cite{Arnold1969}, \cite{Brieskorn1973}, \cite{OrSo1980},
\cite{OrSo1980-2}, \cite{Lehrer1995}, \cite{BlLe2001}).

In particular, Orlik and Solomon \cite[Corollary 5.7]{OrSo1980} showed that $H^{\ast} (\mathcal{M}_{\mathcal{A}(G)}, \mathbb{C})$ is determined (as a graded representation of $G$) by the poset $\mathcal{L}(\mathcal{A}(G))$ of intersections of the hyperplanes in $\mathcal{A}(G)$.  

The poset $\mathcal{L}(\mathcal{A}(G))$ is
known to coincide with the poset of fixed point subspaces (or
1-eigenspaces) of elements of $G$ (see \cite[Theorem 6.27]{OrTe1992}). This paper is the third in a series (following \cite{Koonin2012-5} and \cite{Koonin2012-2}) which uses the eigenspace theory of Springer and Lehrer (\cite{Springer1974}, \cite{LeSp1999}, \cite{LeSp1999-2}) to study generalisations of $\mathcal{L}(\mathcal{A}(G))$ for arbitrary eigenvalues. 

Whereas the focus in the first two papers was on topological properties of the posets in question, the emphasis of this paper is on representation theory. The main tool is the theory of poset extensions due to Segev and Webb (\cite{SeWe1994}). 

The papers \cite{Koonin2012-5} and \cite{Koonin2012-2} study the structure of a poset we call $\widetilde{\mathcal{S}}_{\zeta}^V (\gamma G)$ in detail, whose elements are eigenspaces of elements of a reflection coset $\gamma G$ in $V$, for fixed eignenvalue $\zeta$, ordered by the reverse of inclusion.  This poset is defined in \S\ref{subsection:taxonomy}.

The main theorem of \cite{Koonin2012-5} - Theorem 1.1 - states that in the case $\gamma = \tmop{Id}$, the poset $\widetilde{\mathcal{S}}_{\zeta}^V (G)$ is Cohen-Macaulay over
  $\mathbb{Z}$. Thus the homology of this poset is concentrated in top dimension.  However the structure of the representation of $G$ on this top homology is difficult to understand.  Indeed the most information we have at present is an exponential generating function for the dimension of the representation, when $G$ is an imprimitive reflection group (see \cite{Koonin2012-2}), as well as explicit computations for the dimensions in the other irreducible cases. (In fact, such exponential generating functions exist also in the more general setting of reflection cosets of imprimitive reflection groups.)

This paper suggests a slight modification of $\widetilde{\mathcal{S}}_{\zeta}^V (\gamma G)$, which we call $\left.\mathcal{U}'\right._{\zeta}^V (\gamma G)$. Again we define this poset in \ref{subsection:taxonomy}. Essentially, it consists of adjoining an additional element to the original poset which lies beneath all but the maximal eigenspaces.  

The motivation for this modification comes from the theory of poset extensions due to Segev and Webb (see \cite{SeWe1994}), which will be explained in \S\ref{subsection:extension}. The new poset $\left.\mathcal{U}'\right._{\zeta}^V (\gamma G)$ is shown to be homotopy equivalent to a bouquet of spheres. The number of spheres can be expressed neatly in terms of the invariant theory of $G$, and the representation on $G$ is shown to be that induced from the action of the normaliser of a maximal eigenspace on that eigenspace (see Corollary \ref{corollary:webb}).

\section{Preliminaries}

\subsection{Homology
Representations}\label{subsection:homologyrepresentations}

This section introduces the application of poset homology to representation
theory. Most of the material can be found in \cite{CuLe1981}, and see also
\cite{Rylands1990}, \cite[\S2.3]{Wachs2004}, \cite{ThWe1991}.

\begin{definition}\label{definition:G-poset}$\;$

\begin{itemize} 
\item[(i)] Suppose $(P, \leqslant)$ is a poset and $G$ a
  group. Call $P$ a $G$-poset if $G$ acts on the elements of $P$, and if for all $g
  \in G$, and all $x, y \in P$,
\[
x \leqslant_P y \quad\mbox{implies}\quad gx \leqslant_p gy .
\]
 That is, $G$ acts as a group of automorphisms of $P$.
\item[(ii)] If $P$ and $Q$ are both $G$-posets and $\phi : P \rightarrow Q$ is an
  order-preserving map such that for all $g \in G, x \in P$
\[
\phi (gx) = g \phi (x),
\]
then $\phi$ is said to be a map of $G$-posets.
\item[(iii)] If $P$ and $Q$ are $G$-posets and $\phi : P \rightarrow Q$ is an
  isomorphism of posets and also a $G$-poset map, then $\phi$ is said to be an
  isomorphism of $G$-posets.
\item[(iv)] If $P$ and $Q$ are $G$-posets and $\phi : P \rightarrow Q$ is a homotopy
  equivalence of posets and also a $G$-poset map, then $\phi$ is said to be a
  $G$-homotopy equivalence.
\item[(v)] If $P$ is a $G$-poset which is $G$-homotopy equivalent to a point, then $P$ is
  said to be $G$-contractible.
  \end{itemize}
\end{definition}

If $P$ is a $G$-poset then $G$ also acts on chains of $P$:
\[
g (x_0 < \cdots < x_k) = (gx_0 < \cdots < gx_n).
\]
Note that there is no degeneracy as $G$ acts as a group of automorphisms of
$P$.  This action of $G$ commutes with the boundary homomorphism. Thus $G$
acts on the homology modules $H_i (P, \mathbb{A})$ and $\widetilde{H}_i (P,
\mathbb{A})$, for all $i$. When the ring $\mathbb{A}$ is a
field, this gives a representation of $G$ for each $i$. Such a representation is known as a {\em homology representation} of $G$. See \cite[\S1]{CuLe1981} for further details.

Suppose $P$ is a $G$-poset, and $Q$ is an $H$-poset. Then the product $P\times Q$ is a $(G \times H)$-poset, with $(G \times H)$-action given by
\[
(g, h) (p.q) = (gp, hq),
\]
where $g \in G, h
\in H, p \in P$ and $q \in Q$. 

Clearly this action generalises to finite products of groups, so that if
$P_j$ is a $G_j$-poset for $j = 1, \ldots, n$ then $P_1 \times \cdots \times
P_n$ is a $(G_1 \times \cdots \times
G_n)$-posets.

\subsection{Taxonomy of Posets}\label{subsection:taxonomy}

The central theme of this paper is the study of homological properties of various posets of eigenspaces associated with unitary reflection groups, and associated homology representations of these groups. This section defines the posets we shall consider.

Let $V$ be a vector space over a field $\mathbb{F}$, and $\zeta \in
\mathbb{F}$. If $x \in \tmop{End} (V)$, recall that $V (x, \zeta)$ is the $\zeta$-eigenspace of
$x$ acting on $V$. That is, $V (x, \zeta) \assign \{v \in V \mid x v = \zeta
v\}.$

\begin{definition}
  \label{definition:S_v(zeta)} Let $\gamma G$ be a reflection coset in $V
  =\mathbb{C}^n$, and $\zeta \in \mathbb{C}^{\times}$ be a complex root of
  unity. Define $\mathcal{S}_{\zeta}^V (\gamma G)$ to be the set $\{V (x,
  \zeta) \mid x \in \gamma G\}$, partially ordered by the reverse of
  inclusion. 
\end{definition}

\begin{remark}\label{remark:action}
  There is a natural action of $G$ on
  $\mathcal{S}_{\zeta}^V (\gamma G)$ which arises from the action of $G$ on
  $V$. If $g \in G$ and $V (\gamma x, \zeta) \in \mathcal{S}_{\zeta}^V
  (\gamma G)$, then $g \cdot V (\gamma x, \zeta) \assign V (g \gamma xg^{- 1},
  \zeta) = V (\gamma g' xg^{- 1}, \zeta)$ for some $g' \in G$, since $\gamma$
  normalises $G$. This action clearly respects the order relation on
  $\mathcal{S}_{\zeta}^V (\gamma G)$, and hence turns $\mathcal{S}_{\zeta}^V
  (\gamma G)$ into a $G$-poset.
\end{remark}

It is known (see \cite[Corollary 3.3]{Koonin2012-5}) that the poset
$\mathcal{S}_{\zeta}^V (\gamma G)$ always has a unique maximal element $\hat{1}$, and
it may or may not have a unique minimal element $\hat{0}$ as
well (the full space $V$, for example). This is important to remember in the definitions of the following
posets, which are modifications of $\mathcal{S}_{\zeta}^V (\gamma G)$:

\begin{definition}
  \label{definition:Stilde_v(zeta)} Define $\widetilde{\mathcal{S}}_{\zeta}^V (\gamma
  G)$ to be the subposet of $\mathcal{S}_{\zeta}^V (\gamma G)$ obtained by
  removing the unique maximal element, as well as the unique minimal element if it
  exists.
\end{definition}

\begin{definition}\label{definition:S'_v(zeta)}
 Define $\left.\mathcal{S}'\right._{\zeta}^V (\gamma G)$ to
  be the poset $\mathcal{S}_{\zeta}^V (\gamma G) \backslash \{ \hat{1} \}$.
\end{definition}

The difference between
$\widetilde{\mathcal{S}}_{\zeta}^V (\gamma G)$ and $\left.\mathcal{S}'\right._{\zeta}^V (\gamma
G)$ is that the former does not contain the unique minimal element of
$\mathcal{S}_{\zeta}^V (\gamma G)$ (if it exists), whereas the latter does.

\begin{definition}\label{definition:T'_v(zeta)}
  Define $\left.\mathcal{T}'\right._{\zeta}^V (\gamma G)$ to
  be the subposet of $\left.\mathcal{S}'\right._\zeta^V (\gamma G)$ consisting of
  eigenspaces which are {\em not} maximal.
\end{definition}

That is, $\left.\mathcal{T}'\right._{\zeta}^V (\gamma G)$ is the subposet of $\left.\mathcal{S}'\right._\zeta^V (\gamma G)$ obtained by deleting elements of rank 0.

\begin{definition}\label{definition:U'_v(zeta)} Define the poset $\left.\mathcal{U}'\right._{\zeta}^V (\gamma G)$ as follows.
 The elements of $\left.\mathcal{U}'\right._{\zeta}^V
  (\gamma G)$ are those of $\left.\mathcal{S}'\right._{\zeta}^{V} (\gamma G)$ together
  with one additional element $\hat{0}_S$. The order relation on
  $\left.\mathcal{U}'\right._{\zeta}^V (\gamma G)$ is the following. Given $x, y \in
  \left.\mathcal{U}'\right._{\zeta}^V (\gamma G)$, $x < y$ if and only if either $x, y
  \in\left. \mathcal{S}'\right._{\zeta}^{ V} (\gamma G)$ and $x < y$ in $\left. \mathcal{S}'\right._{\zeta}^{ V} (\gamma G)$, or $x =\hat{0}_S$ and $y
  \in\left.\mathcal{T}'\right._{\zeta}^V (\gamma G)$.
\end{definition}

The purpose of this construction will become apparent in
{\S}\ref{subsection:extension2}.

In the special case $\zeta
= 1, \gamma = \tmop{Id}$, the posets we have defined take the following form.  The posets $\widetilde{\mathcal{S}}_1^V (G)$ and $\left.\mathcal{T}'\right._1^V (G)$ are equal to the poset $\widetilde{\mathcal{L}(\mathcal{A}(G))}$ (the intersection lattice of the reflecting hyperplanes, with minimal and maximal elements removed), $ \left.\mathcal{S}'\right._1^V (G)$ is this same intersection lattice with only the maximal element removed, while $\left.\mathcal{U}'\right.^V_1 (G)$ is the suspension of $\widetilde{\mathcal{S}}_1^V (G)$ (see Definition \ref{definition:suspension} and Proposition \ref{proposition:US}).

It is clear that $\widetilde{\mathcal{S}}_{\zeta}^V (\gamma G)$,
$\left.\mathcal{S}'\right._{\zeta}^V (\gamma G)$ and $\left.\mathcal{T}'\right._{\zeta}^V (\gamma G)$
are subposets of $\mathcal{S}_{\zeta}^V (\gamma G)$ which are stable under the
action of $G$, and are therefore $G$-posets themselves. Define an action of
$G$ on $\left.\mathcal{U}'\right._{\zeta}^V (\gamma G)$ by letting $G$ act trivially on the
additional element $\hat{0}_S$. This makes $\left.\mathcal{U}'\right._{\zeta}^V (\gamma G)$
into a $G$-poset as well.

\begin{remark}
  The posets which hold the most interest for us are
  $\widetilde{\mathcal{S}}_{\zeta}^V (\gamma G)$ and $\left.\mathcal{U}'\right._{\zeta}^V
  (\gamma G)$. The others may be regarded as intermediary posets, whose
  definition is necessary to facilitate the study of these two.  The papers \cite{Koonin2012-5} and \cite{Koonin2012-2} study the structure of $\widetilde{\mathcal{S}}_{\zeta}^V (\gamma G)$ in detail.  This paper deals with the poset $\left.\mathcal{U}'\right._{\zeta}^V
  (\gamma G)$.
\end{remark}

\section{Poset Extensions}\label{subsection:extension}

The background material on extensions of
$G$-posets comes from \cite{SeWe1994}, and the exposition and notation follows that
paper closely.  Proofs of the results in this section can be found in that paper, and in \cite{Koonin2012}. In this section, all homology is taken over $\mathbb{Z}$ unless otherwise stated.
 
\begin{definition}\label{definition:extension}
  Let $Q$ be a subposet of $P$. The poset $P$ is said
  to be an {\em extension} of $Q$ if $Q$ is an upper order ideal of $P$, and if for all $p
  \in P,$ $Q_{\geqslant p} \neq \emptyset$.

  If $P$ is an extension of $Q$ such that for all $p \in P$ either $p \in Q$ or
  $p$ is a minimal element of $P$, then $P$ is said to be an extension of $Q$ by
  minimal elements.
\end{definition}

\begin{definition}\label{definition:Pq}
  If $P$ is an extension of $Q$, define a new poset $P_Q$ as
  follows.

  The elements of $P_Q$ are those of $P$, together with one additional element
  $\hat{0}_Q$. Given $x, y \in P_Q$, define $x <_{P_Q} y$ if and only
  if either $x, y \in P$ and $x <_P y$, or $x = \hat{0}_Q$ and $y \in Q$.

  Denote by $Q_Q \subseteq P_Q$ the poset with elements $Q \cup \hat{0}_Q$
  and the same order relation as $P_Q$. Thus $Q_Q$ is just $Q$ with a
  minimal element adjoined.
\end{definition}

If $P$ is a $G$-poset and $Q$ is stable under the action of $G$, then $P_Q$
becomes a $G$-poset by letting $G$ act trivially on $\hat{0}_Q$.

Denote the
simiplicial chain group of a poset $P$ at dimension $n$ by $C_n (P)$ ($n
\geqslant 0$), and $\widetilde{C}_n (P)$ the augmented simplicial chain group. Thus $\widetilde{C}_n (P) = C_n (P)$ for $n \geqslant 0$, while $\widetilde{C}_{- 1}
(P) =\mathbb{Z}$. As usual let $Z_n (P)$ ($\widetilde{Z}_n (P)$) be the group
of $n$-cycles, $B_n (P)$ ($\widetilde{B}_n (P))$ the group of $n$-boundaries. 
Given a cycle $z \in \widetilde{Z}_n (P)$ denote by $[z] = z + \widetilde{B}_n (P)$
the corresponding element in $\widetilde{H}_n (P)$.

\begin{proposition}{\cite[Proposition 1.1]{SeWe1994}}\label{proposition:MV} Suppose $P$ is an extension
  of $Q$. Then
\begin{itemize}
\item[(i)] $\Delta P_Q = \Delta P \cup \Delta Q_Q$ and $\Delta P = \Delta Q_Q \cap
  \Delta Q$.
\item[(ii)] There is a long exact Mayer-Vietoris sequence in reduced homology
  given by
\[
\cdots \rightarrow \;\widetilde{H}_n (Q)\; \xrightarrow{\iota_*}\;
  \widetilde{H}_n (P) \;\xrightarrow{\kappa_*} \;\widetilde{H}_n (P_Q)\; \xrightarrow{r}
 \; \widetilde{H}_{n - 1} (Q)\; \rightarrow \cdots
  \]
where $\iota_{\ast}$, $\kappa_{\ast}$ are the maps on homology
  induced by the obvious inclusion maps $\iota$, $\kappa,$ and $r$ is given as
  follows. If $\alpha \in \widetilde{C}_n (P)$ and $\beta \in \widetilde{C}_n (Q_Q)$
  are such that $\partial (\alpha + \beta) = 0$, then $r ([\alpha + \beta)] =
  [\partial \alpha]$, where $\partial$ is the differential map of $P_Q$. If
  $P$ is a $G$-poset and $Q$ is stable under the action of $G$, then then
  Mayer-Vietoris sequence is one of $\mathbb{Z}G$-modules.
  \end{itemize}
\end{proposition}

When $P$ is an extension of $Q$ by miminal elements, it is possible to show
that $\Delta P_Q$ is homotopy equivalent to a wedge of suspensions of certain
other posets. Before describing how this is done, it is necessary to define
poset analogues for some common topological constructions.

\begin{definition}\label{definition:suspension}
  Let $R$ be a poset. Define the {\em suspension} of $R$,
  denoted $\Sigma R$, as follows:

  The elements of $\Sigma R$ are those of $R$, together with two additional
  elements $\hat{0}_R$ and $\hat{0}_R'$. Given $x, y \in \Sigma R$, define
  $x <_{\Sigma R} y$ if and only if $x = \hat{0}_R$ and $y \neq \hat{0}_R'$,
  or $x = \hat{0}_R'$ and $y \neq \hat{0}_R$, or $x, y \in R$ and $x <_R y$.
\end{definition}

In order to describe the action of $G$ on the homology of $R$ and $\Sigma R$,
some more notation is needed. If $s = (r_0 < r_1 < \cdots < r_{n - 1})$ is
an $(n - 1)$-simplex of $R$ and $r < r_0$, define $r \ast s \assign (r < r_0 <
r_1 < \cdots < r_{n - 1})$. If $z \in \widetilde{Z}_{n - 1} (R)$, write $z =
\sum_{i = 1}^m n_i s_i$, where $s_i$ is an $(n - 1)$-simplex of $R$. Define
$\hat{0}_R \ast z \assign \sum_{i = 1}^m n_i ( \hat{0}_R \ast s_i) \in
\widetilde{C}_n (\Sigma R)$, and define $\hat{0}_R' \ast z$ similarly. Also
define $\Sigma (z) \assign \hat{0}_R \ast z - \hat{0}'_R \ast z.$

If $\Delta$ is an abstract
simplicial complex and $\Phi$ is the abstract simplicial complex consisting of
two distinct, isolated vertices $v_1$ and $v_2$, then the suspension of
$\Delta$, denoted $\Sigma (\Delta)$, is the join $\Phi \ast
\Delta$.

\begin{proposition}{\cite[Proposition 2.1]{SeWe1994}}\label{proposition:suspension}) Let $R$ be a
  $G$-poset. Then
\begin{itemize}
\item[(i)] There is a $G$-equivariant homeomorphism $\Delta (\Sigma R) \cong_G
  \Sigma (\Delta R) .$
\item[(ii)] For $n \geqslant 1$, if $z \in \widetilde{Z}_{n - 1} (R)$, then
  $\partial ( \hat{0}_R \ast z) = \partial ( \hat{0}_R' \ast z) = z$, where
  $\partial$ is the differential map of $\Sigma R$. Thus $\Sigma (z) \in
  \widetilde{Z}_n (\Sigma R) .$
\item[(iii)] The map $\widetilde{H}_{n - 1} (R) \rightarrow \widetilde{H}_n
  (\Sigma R)$ given by $[z] \rightarrow [\Sigma (z)]$ is an isomorphism of
  $\mathbb{Z}G$-modules.
  \end{itemize}
\end{proposition}

It is also necessary to define a wedge of suspensions of of a set of posets.
\begin{definition}
  \label{definition:wedgeposet}Suppose $\{R_t \mid t \in \mathcal{T}\}$ is a
  family of posets indexed by some set $\mathcal{T}$. Define the {\em wedge of
  suspensions} of the poset $R_t$, denoted $\bigvee_{t \in \mathcal{T}} \Sigma
  R_t$, as follows.

  The elements of $\bigvee_{t \in T} \Sigma R_t$ are defined to be $\bigcup_{t
  \in \mathcal{T}} (R_t \times \{t\}) \cup \mathcal{T} \cup \{ \hat{0} \}$. 
  Define a partial order on this set as follows. For $t \in \mathcal{T}$
  define $j_t : R_t \times \{t\} \rightarrow R_t$ by $j_t (r, t) = r$.

  If $x, y \in \bigvee_{t \in T} \Sigma R_t$, define $x < y$ if and only if
  one of the following holds:
\begin{itemize}
\item[(i)] there exists $t \in \mathcal{T}$ such that $x, y \in R_t \times
  \{t\}$ and $j_t (x) < j_t (y),$
\item[(ii)] $x = t \in \mathcal{T}$ and $y \in R_t \times \{t\}$,
\item[(iii)] $x = \hat{0}$ and $y \nin \mathcal{T} \cup \{ \hat{0} \}$.
\end{itemize}
\end{definition}
Note that $R_t \times \{t\}$ can be identified with $R_t$. The use of $R_t
\times \{t\}$ is to ensure that all sets are disjoint as $t$ runs through
$\mathcal{T}$. In the following proposition this identification is made.

\begin{proposition}{\cite[Proposition 2.2]{SeWe1994}}  \label{proposition:wedge}
Suppose $\{R_t \mid t \in \mathcal{T}\}$ is a family of
  posets. Then:
\begin{itemize}
\item[(i)] $\Delta ( \bigvee_{t \in T} \Sigma R_t) \cong \bigvee_{t \in
  \mathcal{T}} \Sigma (\Delta R_t),$
\item[(ii)] for $n \geqslant 1$ the map
\[
\mu : \bigoplus_{t \in \mathcal{T}}
  \widetilde{H}_{n - 1} (R_t) \rightarrow \widetilde{H}_n ( \bigvee_{t \in
  \mathcal{T}} \Sigma R_t) 
  \]
defined by $\mu ( \sum_{t \in \mathcal{T}} [z_t]) = \sum_{t \in
  \mathcal{T}} [t \ast z_t - \hat{0} \ast z_t]$ is an isomorphism, where for
  all $t \in \mathcal{T}$, $z_t \in \widetilde{Z}_{n - 1} (R_t) .$
  \end{itemize}
\end{proposition}

Of particular interest is the case when $P$ is an extension of $Q$ by minimal
elements. Let the indexing set $\mathcal{T}$ be the set $\mathcal{M}= P\backslash Q$
of elements in $P$ but not $Q$. By definition, this set consists of minimal
elements of $P$. Also take the posets $R_t$ to be the subsposets $P_{> m}$,
$m \in \mathcal{M}$. Recall that by
definition, $P_{> m} ;=\{x \in P \mid x > m\}.$

If $P$ is a $G$-poset and $Q$ is invariant under the action of $G$, then the
poset $\bigvee_{m \in \mathcal{M}} \Sigma P_{> m}$ admits an action of $G$,
defined by
\begin{alignat*}{2}
g \cdot (p, m) &= (g \cdot p,
g \cdot m)&\quad\mbox{ for $(p, m) \in P_{> m} \times \{m\}$},\\
g \cdot m &= m&\mbox{ for $m \in \mathcal{M}$},\\
g \cdot \hat{0} &= \hat{0}.&
\end{alignat*}
Hence the homology groups of $\bigvee_{m \in \mathcal{M}} \Sigma P_{> m}$
become $\mathbb{Z}G$-modules.

There is also an action of $G$ on the simplicial complex $\bigvee_{m \in
\mathcal{M}} \Sigma (\Delta P_{> m})$. With this action, if $x \in \Sigma
(\Delta P_{> m})$ and $g \in G$, then $g \cdot x \in \Sigma (\Delta
(P_{> g \cdot m})$.

\begin{proposition}{\cite[Proposition 2.3]{SeWe1994}}\label{proposition:wedge2} Suppose that $P$ is an
  extension of $Q$ by minimal elements. Further, suppose that $P$ is a $G$-poset
  and that $Q$ is stable under the action of $Q$. Let $\mathcal{M}= P\backslash Q$. Then
\begin{itemize}
\item[(i)] There is a $G$-equivariant homeomorphism
\[
\Delta \Bigl( \bigvee_{m \in
  \mathcal{M}} \Sigma P_{> m}\Bigr) \cong_G \bigvee_{m \in \mathcal{M}} \Sigma
  (\Delta P_{> m}).
  \]
\item[(ii)] For $n \geqslant 1$ the group $\oplus_{m \in \mathcal{M}}
  \widetilde{H}_{n - 1} (P_{> m})$ acquires the structure of an induced
  $\mathbb{Z}G$-module
\[
\bigoplus_{m \in \mathcal{M}}
  \widetilde{H}_{n - 1} (P_{> m}) \simeq_G \bigoplus_{m \in [G \backslash
  \mathcal{M}]} \tmop{Ind}_{G_m}^G (\widetilde{H}_{n - 1} (P_{> m})),
  \]
where $G_m$ denotes the stabiliser of m in G, and $[G
  \backslash \mathcal{M}]$ denotes the set of G-orbits on $\mathcal{M}$. The
  mapping
\[\mu : \bigoplus_{m \in \mathcal{M}}
  \widetilde{H}_{n - 1} (P_{> m}) \rightarrow \widetilde{H}_n \Bigl( \bigvee_{m \in
  \mathcal{M}} \Sigma P_{> m}\Bigr)
  \]
of Proposition \ref{proposition:wedge} is an isomorphism of
  $\mathbb{Z}G$-modules.
  \end{itemize}
\end{proposition}

Now we describe how the wedge of suspensions construction is useful in the
case when $P$ is an extension of $Q$ by minimal elements. Set $\mathcal{M}=
P\backslash Q$, and define $j : \bigvee_{m \in \mathcal{M}} \Sigma P_{> m} \rightarrow
P_Q$ as follows. Define
\begin{alignat*}{2}
j (x) &= j_m (x)&\quad\mbox{for $x \in
P_{> m} \times \{m\}$, where $j_m$ was defined in Definition
\ref{definition:wedgeposet}},\\
j (m) &= m&\mbox{for $m \in
\mathcal{M}$}\\
j ( \hat{0}) &= \hat{0_{}}_Q .&
\end{alignat*}

\begin{theorem}{\cite[Theorem 2.4]{SeWe1994}}\label{theorem:wedge3} Suppose $P$ is an extension
  of $Q$ by minimal elements. Further, suppose that $P$ is a $G$-poset and that $Q$
  is stable under the action of $G$. Let $\mathcal{M}= P\backslash Q$. Then
\begin{itemize}
\item[(i)] $j : \bigvee_{m \in \mathcal{M}} \Sigma P_{> m} \rightarrow P_Q$
  is a $G$-homotopy equivalence.
\item[(ii)] for $n \geqslant 1$ the map
\[
\mu : \bigoplus_{m \in \mathcal{M}}
  \widetilde{H}_{n - 1} (P_{> m}) \rightarrow \widetilde{H}_n (P_Q)
  \]
  is an isomorphism of $\mathbb{Z}G$-modules, where
\[
\mu ( \sum_{m \in \mathcal{M}} [z_m]) =
  \sum_{m \in \mathcal{M}} [m \ast z_m - \hat{0}_Q \ast z_m],
\]  
where for all $m \in \mathcal{M}$, $z_m \in \widetilde{Z}_{n - 1} (P_{>
  m})$.
  \end{itemize}
\end{theorem}

\section{Extensions for $\mathcal{S}_{\zeta}^V (\gamma
G)$}\label{subsection:extension2}

Adopt the notation of \S \ref{subsection:taxonomy}.

  Note that if we set $P = \left.\mathcal{S}'\right._{\zeta}^{
  V} (\gamma G)$ and $Q =\left.\mathcal{T}'\right._\zeta^V (\gamma G)$ then
  $\left.\mathcal{S}'\right._\zeta^V (\gamma G)$ is an extension of
  $\left.\mathcal{T}'\right._\zeta^V (\gamma G)$ by minimal elements, and the resulting
  poset $P_Q =\left.\mathcal{U}'\right._\zeta^V (\gamma G)$. This construction
explains the motivation for the definition.

It is well known (see \cite[Corollary 5.7]{OrSo1980}) that in the case $\zeta
= 1, \gamma = \tmop{Id}$, the posets $\widetilde{\mathcal{S}}_1^V (G) =
\widetilde{\mathcal{L}(\mathcal{A}(G))}$ play an important role in the theory
of hyperplane complements and in the representation theory of unitary
reflection groups. The poset $\left.\mathcal{U}'\right._\zeta^V (\gamma G)$ is a
natural generalisation of $\widetilde{\mathcal{S}}_1^V (G) =
\widetilde{\mathcal{L}(\mathcal{A}(G)}$ by virtue of the following
proposition:

\begin{proposition}\label{proposition:US}
  Let $V$ be a finite dimensional complex vector space,
  and $G$ a unitary reflection group in $V$. Then $\left.\mathcal{U}'\right.^V_1 (G) = \Sigma
  \widetilde{\mathcal{S}}_1^V (G)$.
  \end{proposition}

  \begin{proof}
 This follows from Definition \ref{definition:suspension}, noting
    that $\mathcal{S}_1^V (G)$ always has a unique minimal element $V$.
  \end{proof}

\begin{corollary}
  Let $V$ be a finite dimensional complex vector space of dimension $n$, and $G$ a
  unitary reflection group acting on $V$. Then $\widetilde{H}_{n - 1}
  (\left.\mathcal{U}'\right._1^V (G)) \simeq_G \widetilde{H}_{n - 2} (
  \widetilde{\mathcal{S}}_1^V (G)) .$
\end{corollary}

\begin{proof}
  This follows immediately from Proposition \ref{proposition:suspension}(iii)
  and Proposition \ref{proposition:US}. 
\end{proof}

For general $\gamma$ and $\zeta$, we have the following theorem:

\begin{theorem}\label{theorem:webb}
Suppose $\gamma G$ is a unitary reflection coset in $V
  =\mathbb{C}^n$. Let $\zeta \in \mathbb{C}^n$ and
  set $\mathcal{M}= \left.\mathcal{S}'\right._\zeta^V
  (\gamma G)  \backslash \left.\mathcal{T}'\right._\zeta^V (\gamma G)$. Thus $\mathcal{M}$ is the set of maximal eigenspaces of
  $\left.\mathcal{S}'\right._\zeta^V (\gamma G)$. Then
\begin{itemize}
\item[(i)] There is a long exact sequence of $\mathbb{Z}G$-modules
\[
\cdots \rightarrow \widetilde{H}_n
  (\left.\mathcal{S}'\right._\zeta^V (\gamma G) ) \xrightarrow{\iota_*} \mathcal{S}_{\zeta}^V(\gamma G) 
  \widetilde{H}_{_n} ( \left.\mathcal{T}'\right._\zeta^V (\gamma G) )
  \xrightarrow{\kappa_*} \widetilde{H}_n ( \left.\mathcal{U}'\right._\zeta^V (\gamma G) )
  \xrightarrow{r} \widetilde{H}_{n - 1} (\left.\mathcal{S}'\right._\zeta^V (\gamma G))
  \rightarrow \cdots
  \]
where 
\[\iota_{\ast} :\widetilde{H}_n (
  \left.\mathcal{S}'\right._\zeta^V (\gamma G)) \rightarrow \widetilde{H}_{_n}
  (\left.\mathcal{T}'\right._\zeta^V (\gamma G) )\] and  \[\kappa_{\ast} : \widetilde{H}_{_n}
  (\left.\mathcal{T}'\right._\zeta^V (\gamma G) ) \rightarrow \widetilde{H}_n (
  \left.\mathcal{U}'\right._\zeta^V (\gamma G))\]
  are the maps on homology
  induced by the obvious inclusion maps, and $r$ is the map defined in Proposition \ref{proposition:MV} with $P={\mathcal{S}'}_\zeta^V (\gamma G)$  and  $Q={\mathcal{T}'}_\zeta^V (\gamma G)$ 
\item[(ii)] $\left.\mathcal{U}'\right._\zeta^V (\gamma G) \simeq \bigvee_{m \in \mathcal{M}}
\Sigma ( \left.\mathcal{S}'\right._\zeta^V (\gamma G) _{> m})$
\item[(iii)] For all $n \geqslant 0$,
\begin{align*}
\widetilde{H}_n (\left.\mathcal{U}'\right._\zeta^V (\gamma G) ) &\simeq_G
\bigoplus_{m \in \mathcal{M}} \widetilde{H}_{n - 1} (\left.\mathcal{S}'\right._\zeta^V
(\gamma G) _{> m})\\
&\simeq_G \bigoplus_{m \in
[G \backslash \mathcal{M}]} \tmop{Ind}_{G_m}^G \widetilde{H}_{n - 1} ( \left.\mathcal{S}'\right._\zeta^V (\gamma G))_{> m}).
\end{align*}
\end{itemize}
\end{theorem}

\begin{proof}
 Part (i) now follows
  directly from Proposition \ref{proposition:MV}(ii), part (ii) from Theorem
  \ref{theorem:wedge3}(i), and part (iii) from Proposition
  \ref{proposition:wedge2}(ii). 
\end{proof}

\begin{corollary}\label{corollary:webb}
 Suppose $\gamma G$ is a unitary reflection coset acting on $V =
  \mathbb{C}^n .$ Let $\zeta$ be a complex $m$-th root of unity, and suppose
  $E$ is a maximal $\zeta$-eigepnspace for $\gamma G$. Let $N (E)$ and $C (E)$ be the
  normaliser and centraliser of $E$, respectively. Let
  $\left.\mathcal{U}'\right._\zeta^V (\gamma G)$ be defined as in Definition
  \ref{definition:U'_v(zeta)}. Then
\begin{itemize}
\item[(i)] The poset $\left.\mathcal{U}'\right._\zeta^V (\gamma G)$
  is homotopy equivalent to a bouquet (wedge) of spheres of dimension $l (
  \left.\mathcal{U}'\right._\zeta^V (\gamma G))$. The number of spheres is equal to
\[
\frac{1}{\left| C(E) \right|}\Bigl( \prod_{d_i : m \nmid d_i} d_i\Bigr) \Bigl( \prod_{{d'_i}^{\ast}}
({d'_i}^{\ast} + 1)\Bigr),
\]
where the $\left.d'\right._i^{\ast}$ are the codegrees of $N (E) / C (E)$.
\item[(ii)] When $m$ is a regular number for $\gamma G$, this number is equal to
\[
\Bigl( \prod_{d_i : m \nmid d_i} d_i\Bigr) \Bigl(
\prod_{d_i^{\ast} : m \mid d_i^{\ast}} (d_i^{\ast} + 1)\Bigr).
\]
\item[(iii)] $\widetilde{H}_{\tmop{top}} ( \left.\mathcal{U}'\right._\zeta^V (\gamma
G))\simeq_G  \tmop{Ind}_{N (E)}^G \widetilde{H}_{\tmop{top}} ( \widetilde{\mathcal{S}}^E_1 (\gamma G))
$.
\end{itemize}
\end{corollary}

\begin{proof}
  For (i), we use (ii) of Theorem \ref{theorem:webb}. We have shown in
  \cite[THeorem 3.1]{Koonin2012-5} that $\mathcal{S}_{\zeta}^V (\gamma
  G)_{\geqslant m}  \cong \mathcal{S}_1^E (N (E) / C (E))$, and
  so $\left.\mathcal{S}'\right._\zeta^V (\gamma G)_{> m} \cong \widetilde{\mathcal{S}}_1^E
  (N (E) / C (E))$. It is known that the latter is homotopy equivalent to a
  bouquet of spheres in dimension $l (\left.\mathcal{S}'\right._\zeta^V (\gamma G)) - 1$,
  and that the number of such spheres is equal to the product of the
  coexponents of $N (E) / C (E)$.
  Hence $\Sigma (\left.\mathcal{S}'\right._\zeta^V (\gamma G)_{> m})$ is homotopy
  equivalent to the same number of spheres, but in dimension $l
  (\left.\mathcal{S}'\right._\zeta^V (\gamma G)) = l ( \left.\mathcal{U}\right.'^V_{\zeta}
  (\gamma G))$. To complete the proof of (i) it therefore suffices to count
  the number of maximal eigenspaces in $\gamma G.$ Recall that $G$ acts
  transitively on the set of maximal eigenspaces of $\gamma G$ (see \cite[Theorem 12.19]{LeTa2009}). The stabiliser of a maximal eigenspace $E$
  is $N \assign N (E)$. Hence the number of maximal eigenspaces is
\[
\frac{\left| G \right|}{\left| N(E) \right|}= \frac{\left|G\right|/\left|C(E)\right|}{\left|N(E)\right|/\left|C(E)\right|} = \frac{1}{\left| C(E) \right|} \Bigl(
\prod_{d_i : m \nmid d_i} d_i\Bigr)
 \]

  since it is known (\cite[Corollary 11.17]{LeTa2009}) that the
  degrees of $N / C$ are precisely those degrees of $G$ which are divisible by
  $m$, and that the order of a unitary reflection group is equal to the
  product of its degrees (\cite[Theorem 2.4]{Springer1974}). The
  statement in (i) now follows.

  For (ii), note that by \cite[Lemma 11.22]{LeTa2009}, $m$ is regular for $\gamma
  G$ precisely when $C =\{1\}$. Note that this lemma is stated for
  reflection groups, but applies equally to reflection cosets. Furthermore,
  in this case the codegrees of $N / C$ are precisely those codegrees of $G$
  which are divisible by $m$ (see \cite[Theorem 11.39]{LeTa2009}).
  
To prove (iii) we use Theorem \ref{theorem:webb}(iii). We need only
  note again that the maximal eigenspaces are all conjugate under the action of $G$
  (\cite[Theorem 12.19]{LeTa2009}), so that there is only one term
  in the direct sum.
\end{proof}

\begin{remark}
  This corollary places the well-known representation of $G$ on the top homology of the lattice of interesting hyperlanes (see \cite{OrSo1980}) into a natural family of representations,
  depending on $m$, the order of $\zeta$.  This representation is the case $\gamma = \tmop{Id}$ and $\zeta = 1$.  One of the advantages of working
  with $\left.\mathcal{U}'\right._{\zeta}^V (\gamma G)$ rather than $\left.\mathcal{S}'\right._\zeta^V
  (\gamma G)$ is that the latter may or may not have a unique minimal element. Since it is necessary to remove any unique minimal element before computing
  homology, the posets $\left.\mathcal{S}'\right._\zeta^V (\gamma G)$ must be treated in
  a non-uniform manner. By contrast, the construction of
  $\left.\mathcal{U}'\right._{\zeta}^V (\gamma G)$ is exactly the same whether or not
  $\left.\mathcal{S}'\right._\zeta^V (\gamma G)$ has a unique minimal element.
\end{remark}

\begin{example}
  Consider the case $G = E_8 = G_{37}$, $\gamma = \tmop{Id}$, $\zeta$ a primitive
  3rd root of unity. The degrees of $G$ are 2, 8, 12, 14, 18, 20, 24, 30,
  and the corresponding codegrees are 0, 6, 10, 12, 16, 18, 22, 28 (see for
  example \cite[Table D.3, p.275]{LeTa2009}. Suppose $E$ is a maximal eigenspace
  among $\{V (g, \zeta) \mid g \in E_8 \}$. By \cite[Corollary 11.17]{LeTa2009}, the degrees of $N (E) / C (E)$ are
  precisely the degrees of $G$ which are divisible by 3 -- namely 12, 18, 24 and 30. Now $N (E) / C (E)$ acts irreducibly on $E$ (by \cite[Theorem 11.38]{LeTa2009}), and hence an inspection of the list
  of irreducible reflection groups reveals that the only possibility is $N (E)
  / C (E) \simeq L_4 = G_{32}$.
  
 By \cite[Proposition 11.14]{LeTa2009}, the maximal eigenspaces all
  have dimension equal to the number of degrees divisible by 3. In this
  case, $\dim (E) = 4.$ Hence $l (\left.\mathcal{U}'\right.^{\mathbb{C}^8}_{\zeta}
  (E_8)) = 3$, and so $\widetilde{H}_j ( \left.\mathcal{U}'\right.^{\mathbb{C}^8}_{\zeta}
  (E_8)) = 0$ for $j \neq 3$, and in particular
  $\left.\mathcal{U}'\right.^{\mathbb{C}^8}_{\zeta} (E_8)$ is homotopy equivalent to a
  bouquet of spheres in dimension 3. Now 3 is a regular number for $E_8$, by
  \cite[Theorem 11.28]{LeTa2009}. Hence by Corollary
  \ref{corollary:webb}(ii), the number of spheres in the bouquet is equal to
  $(2 \ast 8 \ast 14 \ast 20) \ast (1 \ast 7 \ast 13 \ast 19) = 7\, 745\, 920$.  Also, by Corollary \ref{corollary:webb}(iii), $\widetilde{H}_3
  (\left.\mathcal{U}'\right.^{\mathbb{C}^8}_{\zeta} (E_8))\simeq \tmop{Ind}_{L_4}^{E_8} \widetilde{H}_2 (
  \widetilde{\mathcal{S}}^E_1 (L_4)).$

Similarly consider the case $G = E_8$, $\gamma = \tmop{Id}$, $\zeta$ a
  primitive 4th root of unity. If $E$ is a maximal eigenspace then $N (E) /
  C (E) \simeq O_4 = G_{31}$. Again, $\widetilde{H}_i (\left.\mathcal{U}'\right.^{\mathbb{C}^8}_\zeta
  (E_8)) = 0$ for $i \neq 3$, and in particular
  $\left.\mathcal{U}'\right.^{\mathbb{C}^8}_\zeta (E_8))$ is homotopy equivalent to a bouquet
  of spheres in dimension 3. The number of spheres is equal to $(2 \ast 14
  \ast 18 \ast 30) \ast (1 \ast 13 \ast 17 \ast 29) = 63\,488\,880$, and $
  \widetilde{H}_3 (\left.\mathcal{U}'\right.^{\mathbb{C}^8}_\zeta (E_8))\simeq \tmop{Ind}_{O_4}^{E_8} \widetilde{H}_2 (
  \widetilde{\mathcal{S}}^E_1 (O_4)) $.
\end{example}

\section{Acknowledgements}
The author would like to acknowledge his supervisor Gus Lehrer and associate supervisor Anthony Henderson for their encouragement, patience, generosity and enthusiasm.

\bibliographystyle{plain}

\bibliography{bibliography2}

\end{document}